\documentclass[english,reqno]{amsart}
\usepackage{etex}
\usepackage{amsmath,amssymb,amsthm,bbm,mathtools,comment}
\usepackage[shortlabels]{enumitem}
\usepackage[pdftex,colorlinks,backref=page,citecolor=blue]{hyperref}
\hypersetup{pdfpagemode=UseNone,pdfstartview={XYZ null null 1.00}}
\usepackage[mathscr]{euscript}
\usepackage[usenames,dvipsnames]{color}
\usepackage{adjustbox,tikz,calc,graphics,babel,standalone}
\usetikzlibrary{decorations.pathreplacing,calc,arrows,arrows.meta,patterns}
\usepackage[final]{microtype}
\usepackage[numbers]{natbib}
\usepackage{cmtiup}
\usepackage{amsfonts}
\usepackage{graphicx}
\usepackage{caption}
\usepackage{subcaption}
\usepackage{verbatim}
\usepackage{array}
\usepackage[frame,cmtip,arrow,matrix,line,graph,curve]{xy}
\usepackage{graphpap, color, pstricks}
\usepackage{pifont}
\usepackage[final]{microtype}
\usepackage{cmtiup}

\allowdisplaybreaks

\theoremstyle{plain}
\newtheorem{theorem}{Theorem}[section]

\newtheorem{conjecture}[theorem]{Conjecture}

\theoremstyle{remark}

\def\N{\mathbb{N}}

\def\C{\mathcal}

\DeclareMathOperator\ind{ind}

\let\originalleft\left
\let\originalright\right
\renewcommand{\left}{\mathopen{}\mathclose\bgroup\originalleft}
\renewcommand{\right}{\aftergroup\egroup\originalright}

\makeatletter
\def\imod#1{\allowbreak\mkern10mu({\operator@font mod}\,\,#1)}
\makeatother

\begin{document}

\title{Counting independent sets in regular hypergraphs}

\author{J\'{o}zsef Balogh}
\address{Department of Mathematics, University of Illinois, 1409 W. Green Street, Urbana, IL 61801}
\email{jobal@illinois.edu}

\author{B\'{e}la Bollob\'{a}s}
\address{Department of Pure Mathematics and Mathematical Statistics, University of Cambridge, Wilberforce Road, Cambridge CB3\thinspace0WB, UK, {\em and\/}
Department of Mathematical Sciences, University of Memphis, Memphis TN 38152, USA}
\email{b.bollobas@dpmms.cam.ac.uk}

\author{Bhargav Narayanan}
\address{Department of Mathematics, Rutgers University, Piscataway, NJ 08854, USA}
\email{narayanan@math.rutgers.edu}

\date{29 January 2020}

\begin{abstract}
Amongst $d$-regular $r$-uniform hypergraphs on $n$ vertices, which ones have the largest number of independent sets? While the analogous problem for graphs (originally raised by Granville) is now well-understood, it is not even clear what the correct general conjecture ought to be; our goal here is propose such a generalisation. Lending credence to our conjecture, we verify it within the class of `quasi-bipartite' hypergraphs (a generalisation of bipartite graphs that seems natural in this context) by adopting the entropic approach of Kahn.
\end{abstract}

\maketitle
\section{Introduction}
This paper concerns the hypergraph analogue of an old (and now resolved) graph-theoretic problem of Granville (see~\citep{Alon}). Granville raised the following problem: which $d$-regular graphs on $n$ vertices have the maximum number of independent sets? This problem was also considered by Kahn~\citep{Kahn} in the context of the hard-core model, and a complete answer is now available owing to the work of Kahn~\citep{Kahn} and Zhao~\citep{Zhao}: the extremal graphs are precisely those consisting of disjoint copies of the complete bipartite graph $K_{d,d}$. By now, much more is known; see~\citep{Perkins, Reverse, Survey} for a small sample of the literature. 

Here, we shall focus on the more general problem of maximising the number of independent sets in  $r$-uniform hypergraphs (or $r$-graphs, for short). While this is a natural problem, we emphasise that it is not even apparent what the correct conjectural analogue of the complete bipartite graph is; our aim in this short note is to remedy this situation. A brief word about notation: a subset of the vertex set of an $r$-graph is \emph{independent} if it induces no edges, the \emph{degree} of a vertex is the number of edges containing it, and an $r$-graph is \emph{$d$-regular} if each of its vertices has degree $d$.

The following construction will play a fundamental role in our arguments. For $r \ge 2$ and $d \in\N$, let $\C{H}^r_d$ be the $d$-regular $r$-partite $r$-graph on $rd$ vertices whose $d^2$ edges are as follows: mark a subset of the vertex set of order $d$, partition the remaining $(r-1)d$ vertices into $d$ sets of $r-1$ vertices each, thereby obtaining an $(r-1)$-uniform matching, and then include in the edge set of $\C{H}^r_d$ each $r$-set consisting of a marked vertex and a matching edge. For example, $\C{H}^2_d$ is the complete bipartite graph $K_{d,d}$, and $\C{H}^3_d$ is the set of triangles in the graph on $3d$ vertices where $d$ vertices are each joined to both ends of all the edges of a matching covering the other $2d$ vertices. 

Writing $\ind(\C{G})$ for the number of independent sets in an $r$-graph $\C{G}$, an easy computation tells us that 
\[\ind(\C{H}^r_d) = 2^{(r-1)d} + (2^d-1)(2^{r-1}-1)^{d}.\] 
Our main reason for writing this note is to make the following conjecture.
\begin{conjecture}\label{mainconj}
	For all $r \ge 2$ and $d \in \N$, if $\C{G}$ is a $d$-regular $r$-graph on $n$ vertices, then 
	\[\ind(\C{G}) \le \ind(\C{H}^r_d)^{n/rd}.\]
\end{conjecture}
First, by way of orientation, let us mention that when $r \ge 3$, a disjoint union of copies of $\C{H}^r_d$ has strictly more independent sets than a comparable disjoint union of copies of a complete $r$-partite $r$-graph (the natural first guess), and of an $r$-partite transversal design (the natural guess within the class of linear $r$-graphs); of course, strictly speaking, we mean this when the numerics allow the latter two constructions (i.e., when $d = t^{r-1}$ for some $t \in \N$ for complete $r$-partite $r$-graphs, and $d$ sufficiently large  for $r$-partite transversal designs). Second, as remarked earlier, Conjecture~\ref{mainconj} for $r = 2$ is the aforementioned Kahn--Zhao theorem, but we are unable to verify it for $r\ge3$; nevertheless, in the spirit of Kahn~\citep{Kahn}, we shall verify our conjecture for `quasi-bipartite' $r$-graphs when $r \ge 3$.

It is worth mentioning that there has been some recent (independent) interest around finding a statement in the spirit of Conjecture~\ref{mainconj}; for example, Cohen, Perkins, Sarantis and Tetali~\citep{linear} study an analogue of the problem treated here for regular \emph{linear} $r$-graphs, and raise the question of what one can say about regular $r$-graphs in general.

This paper is organised as follows. In Section~\ref{quasi}, we prove Conjecture~\ref{mainconj} within the class of quasi-bipartite $r$-graphs, and conclude in Section~\ref{conc} by discussing the main obstacles a proof of Conjecture~\ref{mainconj} would have to overcome.

\section{Quasi-bipartite hypergraphs}\label{quasi}
Recall that, in an $r$-graph $\C{G}$, the \emph{link $\C{L}(v)$} of a vertex $v$ is the $(r-1)$-graph whose edges are precisely those sets $S$ such that $S \cup \{v\}$ is an edge of $\C{G}$ (and whose vertex set is precisely the span of these edges). We say that an $r$-graph $\C{G}$ is \emph{quasi-bipartite} if its vertices may be partitioned into two sets $A$ and $B$ such that 
\begin{enumerate}[(I)]
\item\label{condA} every edge of $\C{G}$ intersects $A$ in exactly one vertex, and
\item\label{condB} for each $a\in A$, the link $\C{L}(a)$ of $a$ is a matching.
\end{enumerate} 
Two remarks about this definition are worth recording: first, when $r=2$, this is easily seen to be precisely the definition of a bipartite graph, and second, we note that $\C{H}^r_d$ is quasi-bipartite for all $r\ge 2$ and $d \in \N$. Following Kahn~\citep{Kahn}, we prove Conjecture~\ref{mainconj} for quasi-bipartite hypergraphs.

\begin{theorem}
If $\C{G}$ is a $d$-regular quasi-bipartite $r$-graph on $n$ vertices, then 
\[ \ind(\C{G}) \le \ind(\C{H}^r_d)^{n/rd}. \]
\end{theorem}
\begin{proof}
Denote by $A\cup B$ the vertex partition associated with $\C{G}$. Let $X$ be the characteristic vector of a randomly chosen independent set of $\C{G}$. For a set $S$ of vertices, we write $X_S$ for the subvector of $X$ indexed by the vertices in $S$, and abbreviate $X_{\{v\}}$ by $X_v$. Now, representing $X=(X_A, X_B)$ and writing $H(X)$ for the entropy of $X$, we note that
\[ \log (\ind(\C{G})) = H(X) = H\left( X_B\right) + H\left(X_A\,|\,X_B\right) ,\]
and observe that $\{V(\C{L}(a)):a\in A\}$ is a $d$-covering of $B$; this fact follows from Condition~\ref{condB}. Hence, by Shearer's lemma (see~\citep{cgfs}), we have
\[H\left(X_B\right) \le \frac{1}{d}\sum_{a\in A} H\left(X_{V(\C{L}(a))}\right).\]
Next, note also that
\[ H\left(X_A\,|\,X_B\right)\le \sum_{a\in A} H\left(X_{a}\,|\,X_B\right) = \sum_{a\in A} H\left(X_{a} \,|\, X_{V(\C{L}(a))}\right).\]
Finally, putting these estimates together, we conclude that
\[ H(X) \le \frac{1}{d} \left(\sum_{a\in A}H\left(X_{V(\C{L}(a))}\right)+  d\cdot H\left(X_{a} \,|\, X_{V(\C{L}(a))}\right)\right).\]

Now, fix an arbitrary $a\in A$ and consider a set $I \subset V(\C{L}(a))$ that is independent in $\C{G}$. We write $p(I)$ for the probability that $X_{V(\C{L}(a))} = I$, and $\lambda(I)$ for the number of ways that $a$ can be added to $I$ (which is either $1$ or $2$) whilst preserving independence in $\C{G}$. In this language, we have
\[ H\left(X_{V(\C{L}(a))}\right)+  d\cdot H\left(X_{a} \,|\, X_{V(\C{L}(a))}\right)= \sum_I  p(I) \left( \log\left(\frac{1}{p(I)}\right) + d \cdot H\left(X_a\,|\,\{X_{V(\C{L}(a))}= I\}\right)\right).\]
Since $H(X_a \,|\,  \{X_{V(\C{L}(a))}= I\})\le \log(\lambda(I))$, the right hand side of the above inequality is bounded above by 
\[\sum_I  p(I) \left( \log\frac{1}{p(I)} + d \cdot \log(\lambda(I))  \right) =  \sum_I  p(I) \log \left( \frac{\lambda(I)^d}{p(I)} \right)\le \log\left(\sum_I \lambda(I)^d \right),\]
where the last inequality is a consequence of Jensen's inequality. Noting that $I$ ranges over the subsets of $V(\C{L}(a))$ that are independent in $\C{G}$, we have
\[\sum_I \lambda(I)^d \le 2^{|V(\C{L}(a))|} + (2^d-1) \ind(\C{L}(a)) \le 2^{(r-1)d} + (2^d-1)(2^{r-1}-1)^d =  \ind(\C{H}^r_d).\]
For the first inequality above, notice that each subset of $V(\C{L}(a))$ contributes at most $1$ to the sum on the left, unless it happens to be independent in $\C{L}(a)$, in which case it contributes an additional $2^d-1$. To see why the second inequality above holds, first observe that we trivially have $|V(\C{L}(a)) |\le (r-1)d$. Then, observe that we may bound $\ind(\C{L}(a))$ again using the subadditivity of entropy (the trivial case of Shearer's lemma): writing $S_1, S_2, \dots, S_d$ for the $d$ edges of the $(r-1)$-graph $\C{L}(a)$, consideration of the characteristic vector $Y$ of a randomly chosen independent set of $\C{L}(a)$ leads us to conclude that $H(Y) \le \sum_{i=1}^d H(Y_{S_i})$, and observing that $H(Y_{S_i}) \le \log(2^{r-1} -1)$ for each $1\le i \le d$ yields the required bound.

Finally, putting these estimates together, and using the fact that $|A| = n/r$ since $\C{G}$ is $d$-regular, allows us to conclude that
\[H(X)\le \frac{n}{rd} \log (\ind(\C{H}^r_d)),\]
and the result follows.
\end{proof}

\section{Conclusion}\label{conc}
There seem to be two major obstacles in adapting our arguments here for quasi-bipartite $r$-graphs to deal with general $r$-graphs; we outline these below briefly.

First, while any entropic approach would seem to demand something like Condition~\ref{condA}, Condition~\ref{condB} appears to be an artefact of our proof; this latter condition allows us to apply Shearer's inequality, and while Shearer's inequality itself cannot be strengthened in its full generality to deal with non-uniform covers (in that entropy may concentrate on the vertices covered the fewest number of times), some variant of Shearer's inequality tailored to the situation at hand might nevertheless be an ingredient that we are presently missing.

Second, a significant difference between $r = 2$ and $r \ge 3$ is that $\C{H}^r_d$ is no longer vertex transitive when $r \ge 3$. Any analogue of the swapping trick of Zhao~\citep{Zhao} for $r$-graphs with $r \ge 3$ would necessarily have to account for this lack of symmetry; in particular, the appropriate `lift' would have to map $\C{H}^r_d$ to disjoint copies of $\C{H}^r_d$ and, in the absence of symmetry, this seems difficult to accomplish .

\section*{Acknowledgements}
The first author was partially supported by NSF grant DMS-1764123, an Arnold O. Beckman Research Award (UIUC Campus Research Board 18132) and the Langan Scholar Fund (UIUC), the second author was supported by NSF grant DMS-1855745, and the third author wishes to acknowledge support from NSF grant DMS-1800521. 

Part of this work was done while the first author was a visiting fellow commoner at Trinity College, Cambridge; we thank Trinity College for their hospitality.

\bibliographystyle{amsplain}
\bibliography{ind_hyp}

\end{document}